\documentclass[reqno]{amsart}
\usepackage[hidelinks]{hyperref}
\usepackage[notref,notcite]{showkeys}
\usepackage{amssymb}
\usepackage{enumerate}

\begin{document}

\title[Semilinear nonlocal differential equations]
{Regularity and stability analysis for a class of semilinear nonlocal differential equations in Hilbert spaces}

\author[T.D. Ke, N.N. Thang, L.T.P. Thuy ]{Tran Dinh Ke$^{\natural}$, Nguyen Nhu Thang, Lam Tran Phuong Thuy}

\address{Tran Dinh Ke  \hfill\break
Department of Mathematics, Hanoi National University of Education \hfill\break
136 Xuan Thuy, Cau Giay, Hanoi, Vietnam}
\email{ketd@hnue.edu.vn} 

\address{Nguyen Nhu Thang  \hfill\break
Department of Mathematics, Hanoi National University of Education \hfill\break
136 Xuan Thuy, Cau Giay, Hanoi, Vietnam}
\email{thangnn@hnue.edu.vn} 

\address{Lam Tran Phuong Thuy  \hfill\break
Department of Mathematics, Electric Power University,\hfill\break
235 Hoang Quoc Viet, Hanoi, Vietnam}
\email{thuyltp@epu.edu.vn} 

\thanks{$^\natural$ Corresponding author. Email: ketd@hnue.edu.vn (T.D.Ke)}
\maketitle
\numberwithin{equation}{section}
\newtheorem{theorem}{Theorem}[section]
\newtheorem{lemma}[theorem]{Lemma}
\newtheorem{proposition}[theorem]{Proposition}
\newtheorem{corollary}[theorem]{Corollary}
\newtheorem{definition}{Definition}[section]
\newtheorem{remark}{Remark}[section]

\begin{abstract}
We deal with a class of semilinear nonlocal differential equations in Hilbert spaces which is a general model for some anomalous diffusion equations. By using the theory of integral equations with completely positive kernel together with local estimates, some existence, regularity and stability results are established. An application to nonlocal partial differential equations is shown to demonstrate our abstract results. 
\end{abstract}

\section{Introduction}
Let $H$ be a separable Hilbert space. Consider the following problem
\begin{align}
\frac{d}{dt}[k*(u-u_0)](t) + Au(t) & = f(u(t)), \; t>0,\label{e1}\\
u(0) & = u_0,\label{e2}
\end{align}
where the unknown function $u$ takes values in $H$, the kernel $k\in L^1_{loc}(\mathbb R^+)$, $A$ is an unbounded linear operator, and $f:H\to H$ is a given function. Here $*$ denotes the Laplace convolution, i.e., $(k*v)(t) = \int_0^t k(t-s)v(s)ds$. 

It should be mentioned that, nonlocal equations have been employed to model different problems related to processes in materials with memory (see, e.g., \cite{CN1981,Grip1985,JW2004,Pruss}). In particular, when the kernel $k(t)=g_{1-\alpha}(t):={t^{-\alpha}}/{\Gamma(1-\alpha)}, \alpha\in (0,1)$, equation \eqref{e1} is in the form of fractional differential equations as the term $\dfrac{d}{dt}[k*(u-u_0)]$ represents the Caputo fractional derivative of order $\alpha$, and this equation has been a subject of an extensive study. In a specific setting, for example, when $H=L^2(\Omega), \Omega\subset\mathbb R^N$, and $A=-\Delta$ is the Laplace operator associated with a boundary condition of Dirichlet/Neumann type, equation \eqref{e1} with a class of kernel functions is utilized to describe anomalous diffusion phenomena including slow/ultraslow diffusions, which were remarked in \cite{VZ}. 

Our motivation for the present work is that, up to our knowledge, no attempt has been made to establish regularity results for \eqref{e1}-\eqref{e2}. Moreover, the stability analysis in the sense of Lyapunov for \eqref{e1} has been less known. In the special case when $k=g_{1-\alpha}$, we refer to some results on stability analysis given in \cite{AK,KL,KL2}. In the recent paper \cite{VZ2017}, Vergara and Zacher investigated a concrete model of type \eqref{e1}, which is a nonlocal semilinear partial differential equation (PDE). Using a maximum principle for the linearized equation, they proved the asymptotic stability for zero solution of this equation. It is worth noting that, the technique used in \cite{VZ2017} does not work for the abstract equation \eqref{e1}. In this paper, the stability of solutions to \eqref{e1} will be analyzed by using a new representation of solutions together with a new Gronwall type inequality. In order to deal with \eqref{e1}, we make the following standing hypotheses.

\begin{itemize}
\it
\item[(\textsf{A})] The operator $A:D(A)\to H$ is densely defined, self-adjoint, and positively definite with compact resolvent.
\item[(\textsf{K})] The kernel function $k\in L^1_{loc}(\mathbb R^+)$ is nonnegative and nonincreasing, and there exists a function $l\in L^1_{loc}(\mathbb R^+)$ such that $k*l=1$ on $(0,\infty)$. 
\item[(\textsf{F})] The nonlinear function $f:H\to H$ is locally Lipschitzian, i.e., for each $\rho>0$ there is a nonnegative number $\kappa (\rho)$ such that
$$
\|f(v_1)-f(v_2)\|\le \kappa(\rho) \|v_1-v_2\|, \;\forall v_1, v_2\in B_\rho,
$$
where $B_\rho$ is the closed ball in $H$ with center at origin and radius $\rho$.
\end{itemize}

Noting that, the hypothesis (\textsf K) was used in a lot of works, e.g. \cite{Koc,KSVZ,SC,VZ,VZ2017,Zacher}. This enables us to transform equations of type \eqref{e1} to a Volterra integral equation with completely positive kernel, which is a main subject discussed in \cite{Pruss}. In this case, one writes $(k,l) \in \mathcal {PC}$. Some typical examples of $(k,l)$ were given in \cite{VZ}, e.g.,
\begin{itemize}
\item $k(t) = g_{1-\alpha}(t)$ and $l(t) = g_\alpha(t), t>0$: slow diffusion (fractional order) case.
\item $\displaystyle k(t) = \int_0^1 g_{\beta}(t)d\beta$ and $\displaystyle l(t) = \int_0^\infty\frac{e^{-pt}}{1+p}dp, t>0$: ultra-slow diffusion (distributed order) case.
\item $\displaystyle k(t) = g_{1-\alpha}(t)e^{-\gamma t}, \gamma\ge 0$, and $\displaystyle l(t) = g_\alpha(t)e^{-\gamma t} + \gamma\int_0^t g_\alpha(s)e^{-\gamma s}ds, t>0$: tempered fractional order case.
\end{itemize}
For more examples on (\textsf K), we refer the reader to \cite{SC}.

Owing these hypotheses, we are able to derive, in the next section, a variation-of-parameter formula as well as the concept of mild solution for inhomogeneous equations. We show that a mild solution is also a weak solution, and it is classical if the external force function is H\"older continuous and the kernel function $l$ is smooth enough. Section 3 is devoted to the semilinear equations, in which we prove the local/global solvability and asymptotic stability for \eqref{e1}. In addition, we show that, the mild solution of semilinear problem is also H\"older continuous. Consequently, we present in the last section an application of the abstract results.

\section{Preliminaries}

For $\mu\in \mathbb R^+$, consider the following scalar integral equations
\begin{align}
& s(t) + \mu (l*s)(t) =1, \; t\ge 0,\label{eq-s}\\
& r(t) + \mu (l*r)(t) = l(t), \; t>0.\label{eq-r}
\end{align}
In the sequel, we assume, in addition, that $l$ is continuous on $(0,\infty)$. Under this assumption, the existence and uniqueness of $s$ and $r$ were examined in \cite{Miller}. In the case $l(t)=g_\alpha(t)$, following from the Laplace transform of $s(\cdot)$ and $r(\cdot)$, we know that $s(t)=E_{\alpha,1}(-\mu t^\alpha)$ and $r(t)= t^{\alpha-1}E_{\alpha,\alpha}(-\mu t^\alpha)$, here $E_{\alpha,\beta}$ is the Mittag-Leffler function defined by
\begin{align*}
E_{\alpha,\beta}(z)=\sum_{n=0}^\infty\frac{z^n}{\Gamma(\alpha n+\beta)},\; z\in \mathbb C.
\end{align*} 
Recall that the kernel function $l$ is said to be completely positive iff $s(\cdot)$ and $r(\cdot)$ take nonnegative values for every $\mu>0$. The complete positivity of $l$ is equivalent to that (see \cite{CN1981}), there exist $\alpha\ge 0$ and $k\in L^1_{loc}(\mathbb R^+)$ nonnegative and nonincreasing which satisfy $\alpha l+ l*k=1$. In particular, the hypothesis (\textsf K) ensures that $l$ is completely positive.

Denote by $s(\cdot,\mu)$ and $r(\cdot,\mu)$ the solutions of \eqref{eq-s} and \eqref{eq-r}, respectively. We collect some properties of these functions.

\begin{proposition}\label{pp-sr}
Let the hypothesis (\textsf K) hold. Then for every $\mu>0$, $s(\cdot,\mu),r(\cdot,\mu)\in L^1_{loc}(\mathbb R^+)$. In addition, we have:
\begin{enumerate}
\item The function $s(\cdot,\mu)$ is nonnegative and nonincreasing. Moreover, 
\begin{equation}\label{eq-s1}
s(t,\mu)\left[1+\mu\int_0^t l(\tau)d\tau\right]\le 1, \;\forall t\ge 0.
\end{equation}
\item The function $r(\cdot, \mu)$ is nonnegative and the following relations hold
\begin{align}
& s(t,\mu) =1-\mu \int_0^t r(\tau,\mu)d\tau= k*r(\cdot,\mu)(t), \; t\ge 0.\label{eq-sr1}
\end{align}
\item For each $t>0$, the functions $\mu\mapsto s(t,\mu)$ and $\mu\mapsto r(t,\mu)$ are nonincreasing.
\end{enumerate}
\end{proposition}
\begin{proof}
The justification for \eqref{eq-s1} and \eqref{eq-sr1} can be found in \cite{CN1981}. We prove the last statement. For $\beta\in L^1_{loc}(\mathbb R^+)$, we denote by $\hat\beta$ the Laplace transform of $\beta$. It follows from \eqref{eq-s}-\eqref{eq-r} that
\begin{align*}
\hat s(\lambda,\mu) & = \lambda^{-1}(1+\mu\hat l)^{-1},\\
\hat r(\lambda,\mu) & = \hat l (1+\mu\hat l)^{-1}, \; \lambda>0.
\end{align*}
Then
\begin{align*}
\frac{\partial}{\partial\mu}\hat s(\lambda,\mu) & = - \lambda^{-1}\hat l (1+\mu\hat l)^{-2} = - \hat s(\lambda,\mu) \hat r(\lambda,\mu),\\
\frac{\partial}{\partial\mu}\hat r(\lambda,\mu) & = - \hat l^2 (1+\mu\hat l)^{-2}=-\hat r(\lambda,\mu)\hat r(\lambda,\mu).
\end{align*}
Taking the inverse transform and using the convolution rule, we get
\begin{align*}
\frac{\partial}{\partial\mu}s(t,\mu) & = - s(\cdot,\mu)*r(\cdot,\mu)(t)\le 0,\\
\frac{\partial}{\partial\mu}r(t,\mu) & = - r(\cdot,\mu)*r(\cdot,\mu)(t)\le 0, \; t>0.
\end{align*}
The proof is complete.
\end{proof}
\begin{remark}\label{rm-sr}
\begin{enumerate}
\item As mentioned in \cite{VZ2017}, the functions $s(\cdot,\mu)$ and $r(\cdot,\mu)$ take nonnegative values even in the case $\mu\le 0$. 
\item Equation \eqref{eq-s} is equivalent to the problem
\begin{align*}
\frac{d}{dt}[k*(s-1)] + \mu s &= 0,\;  s(0) =1.
\end{align*}
This can be seen by convoluting both side of equation $[s-1] + \mu l*s =0$ with $k$ and using $k*l=1$. 
\item Let $v(t) = s(t,\mu)v_0 + (r(\cdot,\mu)*g)(t)$, here $g\in L^1_{loc}(\mathbb R^+)$. Then $v$ solves the problem
\begin{align*}
\frac{d}{dt}[k*(v-v_0)](t) + \mu v(t) & = g(t),\; v(0)  = v_0.
\end{align*} 
Indeed, by formulation and the relation $k*r=s$, we have 
\begin{align*}
k*(v-v_0) & =  k*(s-1)v_0 + k*r*g\\
& = k*(s-1)v_0 + s*g.
\end{align*}
So
\begin{align*}
\frac{d}{dt}[k*(v-v_0)] & = \frac{d}{dt}[k*(s-1)]v_0 + s(0,\mu)g + s'(\cdot,\mu)*g\\
& = -\mu  s(\cdot,\mu) v_0 + g -\mu r(\cdot,\mu)*g\\
& = -\mu [s(\cdot,\mu)v_0 + r(\cdot,\mu)*g] + g\\
& = -\mu v + g,
\end{align*}
thanks to the fact that $s(0,\mu)=1$ and $s'(t,\mu) = -\mu r(t,\mu), \; t>0$.
\item We deduce from \eqref{eq-s1} that, if $l\not\in L^1(\mathbb R^+)$ then $\lim\limits_{t\to\infty}s(t,\mu)=0$ for every $\mu>0$. 
\item It follows from \eqref{eq-sr1} that $\int_0^t r(\tau,\mu)d\tau \le \mu^{-1}, \;\forall t>0$. If $l\not\in L^1(\mathbb R^+)$ then $\int_0^\infty r(\tau,\mu)d\tau = \mu^{-1}$ for every $\mu>0$.
\end{enumerate}
\end{remark}
We are now in a position to prove a Gronwall type inequality, which play an important role in our analysis.
\begin{proposition}\label{pp-gronwall}
Let $v$ be a nonnegative function satisfying
\begin{equation}
v(t) \le s(t,\mu)v_0 + \int_0^t r(t-\tau, \mu) [\alpha v(\tau) + \beta(\tau)]d\tau, \; t\ge 0,\label{pp-gronwall-0}
\end{equation}
for $\mu>0, v_0\ge 0, \alpha>0$ and $\beta\in L^1_{loc}(\mathbb R^+)$. Then
$$
v(t) \le s(t, \mu-\alpha)v_0 + \int_0^t r(t-\tau,\mu-\alpha)\beta(\tau)d\tau.
$$
Particularly, if $\beta$ is constant then
$$
v(t) \le s(t, \mu-\alpha)v_0 +\frac{\beta}{\mu-\alpha}(1-s(t,\mu-\alpha)).
$$
\end{proposition}
\begin{proof}
Let $w(t)$ be the expression in the right hand side of \eqref{pp-gronwall-0}. Then $v(t)\le w(t)$ for $t\ge 0$, and $w$ solves the problem
\begin{align*}
\frac{d}{dt}[k*(w-v_0)](t) + \mu w(t) & = \alpha v(t) + \beta(t),\\
w(0) & = v_0,
\end{align*}
thanks to Remark \ref{rm-sr} (2). This is equivalent to 
\begin{align*}
\frac{d}{dt}[k*(w-v_0)](t) + (\mu -\alpha) w(t) & = \alpha (v(t)-w(t)) + \beta(t),\\
w(0) & = v_0,
\end{align*}
which implies
\begin{align*}
w(t) & = s(t,\mu-\alpha)v_0 + \int_0^t r(t-\tau,\mu-\alpha)[\alpha (v(\tau)-w(\tau)) + \beta(\tau)]d\tau\\
& \le s(t,\mu-\alpha)v_0 + \int_0^t r(t-\tau,\mu-\alpha)\beta(\tau) d\tau,
\end{align*}
in accordance with $v(\tau)-w(\tau)\le 0$ for $\tau\ge 0$ and the positivity of $r$. 

Finally, if $\beta$ is constant, we employ \eqref{eq-sr1} to get 
$$
\int_0^t r(t-\tau,\mu-\alpha)\beta d\tau = \beta \int_0^t r(\tau,\mu-\alpha) d\tau=\frac{\beta}{\mu-\alpha}(1-s(t,\mu-\alpha)),
$$
which completes the proof.
\end{proof}
Let us mention that, the hypothesis (\textsf A) ensures the existence of an orthonormal basis of $H$ consisting of eigenfunctions $\{e_n\}_{n=1}^\infty$ of the operator $A$ and we have
$$
Av = \sum_{n=1}^\infty \lambda_n v_n e_n,
$$
where $\lambda_n>0$ is the eigenvalue corresponding to the eigenfunction $e_n$ of $A$,
$$
D(A) = \{v = \sum_{n=1}^\infty v_ne_n: \sum_{n=1}^\infty \lambda_n^2v_n^2 <\infty\}.
$$
We can assume that $0<\lambda_1\le \lambda_2\le ... \le \lambda_n \to \infty$ as $n\to\infty$.

For $\gamma\in\mathbb R$, one can define the fractional power of $A$ as follows
\begin{align*}
D(A^\gamma) &= \left\{v = \sum_{n=1}^\infty v_ne_n: \sum_{n=1}^\infty \lambda_n^{2\gamma}v_n^2<\infty \right\},\\
A^\gamma v &= \sum\limits_{n=1}^\infty \lambda_n^\gamma v_n e_n.
\end{align*}
Let $V_\gamma = D(A^\gamma)$. Then $V_\gamma$ is a Banach space endowed with the norm
$$
\|v\|_\gamma = \|A^\gamma v\| = \left(\sum_{n=1}^\infty \lambda_n^{2\gamma}v_n^2\right)^{\frac 12}.
$$
Furthermore, for $\gamma>0$, we can identify the dual space $V_\gamma^*$ of $V_\gamma$ with $V_{-\gamma}$.

We now define the following operators
\begin{align}
S(t)v &= \sum_{n=1}^\infty s(t,\lambda_n) v_n e_n,\; t\ge 0, v\in H,\label{op-S}\\
R(t)v &= \sum_{n=1}^\infty r(t,\lambda_n) v_n e_n,\; t>0, v\in H.\label{op-R}
\end{align}
It is easily seen that $S(t)$ and $R(t)$ are linear. We show some basic properties of these operators in the following lemma.
\begin{lemma}\label{lm-SR}
Let $\{S(t)\}_{t\ge 0}$ and $\{R(t)\}_{t > 0},$ be the families of linear operators defined by \eqref{op-S} and \eqref{op-R}, respectively. Then
\begin{enumerate}
\item For each $v\in H$ and $T>0$, $S(\cdot)v\in C([0,T];H)$ and $AS(\cdot)v\in C((0,T];H)$. Moreover,
\begin{align}
& \|S(t)v\|\le s(t,\lambda_1)\|v\|,\; t\in  [0,T], \label{lm-SR1}\\  
& \|AS(t)v\| \le \frac{\|v\|}{(1*l)(t)},\; t\in (0,T]. \label{lm-SR2}
\end{align}
\item Let $v\in H, T>0$ and $g\in C([0,T];H)$. Then $R(\cdot)v\in C((0,T];H)$, $R*g\in C([0,T];H)$ and $A(R*g) \in C([0,T];V_{-\frac{1}{2}})$. Furthermore,
\begin{align}
& \|R(t)v\| \le r(t,\lambda_1)\|v\|, \; t\in (0,T],\label{lm-SR3a}\\
& \|(R*g)(t)\| \le \int_0^t r(t-\tau,\lambda_1) \|g(\tau)\| d\tau, \; t\in [0,T],\label{lm-SR3b}\\
& \|A(R*g)(t)\|_{-\frac 12}\le  \left(\int_0^t r(t-\tau,\lambda_1) \|g(\tau)\|^2 d\tau\right)^{\frac 12},\; t\in [0,T].\label{lm-SR3c}
\end{align}
\end{enumerate}
\end{lemma}
\begin{proof}
(1) For the first statement, we observe that
\begin{align}\label{lm-SR4}
\|S(t)v\|^2 = \sum_{n=1}^\infty s^2(t,\lambda_n)v_n^2.
\end{align}
Since $s(t,\lambda_n)\le 1$ for every $t\ge 0, n\in\mathbb N$, this series is uniformly convergent on $[0,T]$. So is series \eqref{op-S}. Due to the fact that $s(\cdot, \lambda_n)$ is continuous, we get $S(\cdot)v\in C([0,T];H)$. Estimate \eqref{lm-SR1} is deduced from \eqref{lm-SR4} by using $s(t,\lambda_n)\le s(t,\lambda_1)$ for all $n>1$.

Considering
\begin{equation}\label{lm-SR5}
AS(t)v = \sum_{n=1}^\infty \lambda_n s(t,\lambda_n) v_n e_n,
\end{equation}
we have
\begin{equation}\label{lm-SR5a}
\|AS(t)v\|^2 = \sum_{n=1}^\infty \lambda^2_n s^2(t,\lambda_n) v^2_n.
\end{equation}
In view of \eqref{eq-s1}, we get 
$$
\lambda_n s(t,\lambda_n)  \le \frac{\lambda_n}{1+\lambda_n (1*l)(t)} \le \frac{1}{(1*l)(t)}, \forall t>0.
$$
Substituting into \eqref{lm-SR5a}, we have $\|AS(t)v\| \le \dfrac{\|v\|}{(1*l)(t)}$, for every $t>0$. In addition, for any $\delta$ such that $0<\delta<T$, one has $\lambda_n s(t,\lambda_n)\le \dfrac{1}{(1*l)(\delta)}$, which implies that the convergence of \eqref{lm-SR5a} as well as \eqref{lm-SR5} is uniform on $[\delta, T]$. That is, $AS(\cdot)v\in C([\delta, T];H)$. 

(2) Recall that $r(\cdot, \mu)$ is continuous on $(0,\infty)$ (see, e.g. \cite{Miller}). So for any $\delta \in (0, T)$ and $\mu>0$, $r(\cdot,\mu) \in C([\delta, T])$. This ensures that the series
\begin{equation}\label{lm-SR6}
\|R(t)v\| = \sum_{n=1}^\infty r^2(t,\lambda_n)v_n^2
\end{equation}
is uniformly convergent on $[\delta, T]$. So is series \eqref{op-R}. Inequality \eqref{lm-SR3a} follows from \eqref{lm-SR6} since $r(t,\cdot)$ is nonincreasing. 

We now prove that $R*g\in C([0,T];H)$. Denoting $g_n(t) = (g(t), e_n)$, we first check that
\begin{equation}\label{lm-SR7}
(R*g)(t) = \sum_{n=1}^\infty [r(\cdot,\lambda_n)*g_n] (t) e_n.
\end{equation}
Indeed, since $g\in C([0, T];H)$, the series $\|g(t)\| = \sum\limits_{n=1}^\infty |g_n(t)|^2$ is uniformly convergent on $[\delta, T]$. Given $\epsilon>0$, for $\delta \le \tau\le t\le T$ and $p\in\mathbb N$, we have
\begin{align*}
\|\sum_{k=n}^{n+p} r(\tau,\lambda_k)g_k(t-\tau) e_k\| & \le r(\tau,\lambda_1)\left(\sum_{k=n}^{n+p}|g_k(t-\tau)|^2\right)^{\frac 12}<\epsilon,
\end{align*}
provided that $n$ is large enough. So the series $\sum\limits_{n=1}^{\infty} r(\tau,\lambda_n)g_n(t-\tau) e_n$ converges uniformly on $[\delta, T]$ and one can take integration term by term on $[\delta, t]$, i.e. 
$$
\int_\delta^t R(t-\tau)g(\tau)d\tau = \sum_{n=1}^\infty\left( \int_\delta^t r(\tau,\lambda_n)g_n(t-\tau)d\tau \right)e_n.
$$
Fix $t>0$ and put $h_n(\delta) = \int\limits_\delta^t r(\tau,\lambda_n)g_n(t-\tau)d\tau$. Arguing as above for the uniform convergence of the series $\sum\limits_{n=1}^\infty h_n(\delta)e_n$ on $[0,t]$, we can pass to the limit as $\delta\to 0$ to get \eqref{lm-SR7}. Taking \eqref{lm-SR7} into account, by the H\"older inequality, one has
\begin{align}
|[r(\cdot,\lambda_n)*g_n] (t)| & \le \int_0^t \sqrt{r(t-\tau,\lambda_n)}\sqrt{r(t-\tau,\lambda_n)}|g_n(\tau)|d\tau\notag\\
& \le \left(\int_0^t r(t-\tau,\lambda_n)d\tau\right)^{\frac 12}\left(\int_0^t r(t-\tau,\lambda_n) |g_n(\tau)|^2d\tau\right)^{\frac 12}\notag\\
& \le \left(\frac{1}{\lambda_n}(1-s(t,\lambda_n))\right)^{\frac 12}\left(\int_0^t r(t-\tau,\lambda_1) |g_n(\tau)|^2 d\tau\right)^{\frac 12}\notag\\
& \le \frac{1}{\lambda_n^{\frac 12}}\left(\int_0^t r(t-\tau,\lambda_1) |g_n(\tau)|^2d\tau\right)^{\frac 12},\label{lm-SR8}
\end{align}
thanks to \eqref{eq-sr1} and the monotonicity of $r(t,\cdot)$. Then it follows
\begin{align*}
\sum_{k=n}^{n+p}|[r(\cdot,\lambda_k)*g_k] (t)|^2 & \le \frac{1}{\lambda_1} \int_0^t r(t-\tau,\lambda_1)\left(\sum_{k=n}^{n+p}|g_k(\tau)|^2\right)d\tau\\
& \le \frac{\epsilon}{\lambda_1} \int_0^t r(t-\tau,\lambda_1)d\tau\le \frac{\epsilon}{\lambda^2_1},
\end{align*}
for $n$ large, thanks to the uniform convergence of $\sum\limits_{n=1}^\infty |g_n(t)|^2$ on $[0,T]$ and relation \eqref{eq-sr1}. Hence \eqref{lm-SR7} is uniformly convergent on $[0,T]$ and then $R*g\in C([0,T];H)$. Estimate \eqref{lm-SR3b} takes place by employing \eqref{lm-SR3a}.

Finally, we show that $A(R*g) \in C([0,T];V_{-\frac{1}{2}})$. Noticing that 
\begin{align*}
A(R*g)(t) = \sum_{n=1}^\infty \lambda_n [r(\cdot,\lambda_n)*g_n] (t) e_n,
\end{align*}
we obtain
\begin{equation}\label{lm-SR9}
\|A(R*g)(t)\|^2_{-\frac 12} = \|A^{\frac 12}(R*g)(t)\|^2= \sum_{n=1}^\infty \left(\lambda^{\frac{1}{2}}_n [r(\cdot,\lambda_n)*g_n] (t)\right)^2.
\end{equation}
Using estimate \eqref{lm-SR8}, one can claim the uniform convergence of \eqref{lm-SR9} on $[0, T]$ and estimate \eqref{lm-SR3c} follows. Thus $A(R*g)\in C([0,T];V_{-\frac 12})$ as desired.

The proof is complete.
\end{proof}
\begin{remark}\label{rm-SR}
\begin{enumerate}
\item Obviously, $S(0)v=v$ for every $v\in H$.
\item We have $(R*g)(0) = 0$. Indeed, it follows from \eqref{lm-SR3b} that
\begin{align*}
\|(R*g)(t)\| & \le  \sup_{\tau\in [0,T]}\|g(\tau)\|\int_0^t r(t-\tau,\lambda_1)d\tau\\
& =  \sup_{\tau\in [0,T]}\|g(\tau)\|\lambda_1^{-1} (1-s(t,\lambda_1))\to 0\text{ as } t\to 0.
\end{align*}
\item Lemma \ref{lm-SR} implies that $A^{\frac 12}S(\cdot)v, A^{\frac 12} (R*g) \in C((0,T];H)$ for every $v\in H$ and $g\in C([0,T];H)$. Equivalently, $S(\cdot)v, R*g \in C((0,T];V_{\frac 12})$.
\end{enumerate}
\end{remark}

Given $g\in C([0,T];H)$, consider the linear problem
\begin{align}
\frac{d}{dt}[k*(u-u_0)](t) + Au(t) & = g(t), t\in (0,T],\label{le1}\\
u(0) & = u_0.\label{le2}
\end{align}
Based on the operators $S(t)$ and $R(t)$, we introduce the following definition of mild solutions to \eqref{le1}-\eqref{le2}.
\begin{definition}\label{def-mild-sol}
A function $u\in C([0,T];H)$ is called a mild solution to the problem \eqref{le1}-\eqref{le2} on $[0,T]$ iff
\begin{equation}\label{def-mild-sol-0}
u(t) = S(t) u_0 + \int_0^t R(t-s)g(s)ds,\; t\in [0,T].
\end{equation}
\end{definition}
\section{Weak solution and regularity}
\subsection{Existence and uniqueness}
In the sequel, we will define weak solution for \eqref{le1}-\eqref{le2} and show that a mild solution is also a weak solution.
\begin{definition}
A function $u\in C([0,T];H)\cap C((0,T]; V_{\frac 12})$ is said to be a weak solution to \eqref{le1}-\eqref{le2} on $[0,T]$ iff $u(0)=u_0$ and equation \eqref{le1} holds in $V_{-\frac 12}$.
\end{definition}
\begin{theorem}\label{th-wsol}
If $u$ is a mild solution to the problem \eqref{le1}-\eqref{le2}, then it is a weak solution. 
\end{theorem}
\begin{proof}
Let $u$ be defined by \eqref{def-mild-sol-0}. Then Lemma \ref{lm-SR} ensures that $S(\cdot)u_0$ and $R*g$ belong to $C([0,T];H)$, so $u = S(\cdot)u_0 + R*g \in C([0,T];H)$. By Remark \ref{rm-SR}, we get $u(0) = u_0$ and $u\in C((0,T];V_{\frac 12})$.

By formulation, we have 
\begin{align*}
k(\tau)(u(t-\tau)-u_0) & = \sum_{n=1}^\infty k(\tau)[s(t-\tau,\lambda_n) - 1] u_{0n} e_n \\
& + \sum_{n=1}^\infty k(\tau)[r(\cdot,\lambda_n)*g_n](t-\tau) e_n\\
\end{align*}
for $\delta\le \tau\le t\le T$, where $\delta\in (0,T)$, and these series are uniformly convergent on $[\delta, t]$. So one has
\begin{align}
\int_\delta^t k(\tau)(u(t-\tau)-u_0) d\tau & = \sum_{n=1}^\infty \int_\delta^t k(\tau)[s(t-\tau,\lambda_m) - 1] d\tau u_{0n} e_n \notag \\
& + \sum_{n=1}^\infty \int_\delta^t k(\tau)[r(\cdot,\lambda_n)*g_n](t-\tau)d\tau e_n.\label{th-wsol1}
\end{align}
For fixed $t\in (0,T]$, put
$$
h_n(\delta) = \int_\delta^t k(\tau)[s(t-\tau,\lambda_m) - 1] d\tau\,u_{0n} + \int_\delta^t k(\tau)[r(\cdot,\lambda_n)*g_n](t-\tau)d\tau.
$$
Obviously, $h_n$ is continuous on $[0,t]$ for all $n$, and the function $\delta\mapsto h(\delta) = \int_\delta^t k(\tau)(u(t-\tau)-u_0) d\tau$ is also continuous on $[0,t]$. Then the series $\sum_{n=1}^\infty h_n(\delta) e_n$ converges uniformly on $[0,t]$, which enables us to pass to the limit in \eqref{th-wsol1} to obtain
\begin{align}\label{th-wsol2}
k*(u-u_0)(t) & = \sum_{n=1}^\infty k*(s(\cdot,\lambda_n)-1)(t)u_{0n} e_n + \sum_{n=1}^\infty k*[r(\cdot,\lambda_n)*g_n](t) e_n\notag\\
& = \sum_{n=1}^\infty k*(s(\cdot,\lambda_n)-1)(t)u_{0n} e_n + \sum_{n=1}^\infty [s(\cdot,\lambda_n)*g_n](t) e_n,
\end{align}
thanks to \eqref{eq-sr1}. We testify that, it is possible to take differentiation term by term in \eqref{th-wsol2}. It suffices to prove that the series
\begin{equation}\label{th-wsol3}
\sum_{n=1}^\infty \frac{d}{dt}[k*(s(\cdot,\lambda_n)-1)](t)u_{0n} e_n + \sum_{n=1}^\infty \frac d{dt} [s(\cdot,\lambda_n)*g_n](t) e_n
\end{equation}
is uniformly convergent on $[\delta, T]$ for any $\delta\in (0,T)$. Indeed, by Remark \ref{rm-sr} we have
\begin{align*}
\frac{d}{dt}[k*(s(\cdot,\lambda_n)-1)](t) & = -\lambda_n s(t,\lambda_n),\\
\frac d{dt} [s(\cdot,\lambda_n)*g_n](t) & = g_n(t) - \lambda_n [r(\cdot,\lambda_n)*g](t).
\end{align*}
Therefore, \eqref{th-wsol3} becomes
\begin{align*}
& - \sum_{n=1}^\infty \lambda_n s(t,\lambda_n) u_{0n} e_n -\sum_{n=1}^\infty\lambda_n [r(\cdot,\lambda_n)*g](t)e_n + \sum_{n=1}^\infty g_n (t)e_n \\
& = -AS(t)u_0 - A(R*g)(t) + g(t),
\end{align*}
which are uniformly convergent on $[\delta,T]$ as shown in Lemma \ref{lm-SR}. Hence, we can take differentiation in \eqref{th-wsol2} and get the equation
$$
\frac{d}{dt}[k*(u-u_0)](t) = -AS(t)u_0 - A(R*g)(t) + g(t) = -Au(t) + g(t), t\in (0,T],
$$
which holds in $V_{-\frac 12}$. The proof is complete.
\end{proof}
We are in a position to prove the uniqueness of weak solution.
\begin{theorem}\label{th-uniq}
Problem \eqref{e1}-\eqref{e2} has a unique weak solution.
\end{theorem}
\begin{proof}
It remains to show the uniqueness. Let $h_\mu=-s'_\mu=\mu r$, then $h_\mu$ is nonnegative and solves the equation
$$
h_\mu (t) + \mu(h_\mu*l)(t) = \mu l(t), \; t>0, \mu>0.
$$
In addition, for $1\le p<\infty$, $f\in L^p(0,T)$, one has $h_n*f\to f$ in $L^p(0,T)$ as $n\to \infty$ (\cite{Zacher2}). Put $k_\mu = k*h_\mu$, then $k_\mu = \mu k*r=\mu s_\mu$, thanks to \eqref{eq-sr1}. Hence $k_\mu\in W^{1,1}(0,T)$. This enable us to employ the fundamental identity (\cite[Lemma 2.3]{VZ})
\begin{align*}
\left(v(t), (k_\mu*v)'(t)\right) & = \frac 12 (k_\mu*\|v(\cdot)\|^2)'(t) + \frac 12 k_\mu(t)\|v(t)\|^2\\
& + \frac 12 \int_0^t \|v(t)-v(t-s)\|^2[-k'_\mu(s)]ds,\; t\in [0,T], v\in C([0,T];H).
\end{align*}
Therefore
\begin{equation}\label{th-uniq1}
\left(v(t), (k_\mu*v)'(t)\right) \ge \frac 12 (k_\mu*\|v(\cdot)\|^2)'(t), \; t\in [0,T], v\in C([0,T];H),
\end{equation}
thanks to the fact that $k_\mu$ is nonincreasing. 

Let $u_1$ and $u_2$ be weak solutions of \eqref{e1}-\eqref{e2}. Put $v=u_2-u_2$, then we have
\begin{align*}
((k*v)'(t),w) +(Av(t), w) & = 0,\;\forall t\in (0,T],\; w\in V_{\frac 12},\\
v(0) & = 0.
\end{align*}
Then
\begin{align*}
((h_n*k*v)(t),w) +(h_n*1*Av(t), w) & = 0,\;\forall t\in (0,T],\; w\in V_{\frac 12}.
\end{align*}
which is equivalent to
\begin{align*}
((k_n*v)'(t),w) +(h_n*Av(t), w) & = 0,\;\forall t\in (0,T],\; w\in V_{\frac 12}.
\end{align*}
Taking $w=v(t)$ and using \eqref{th-uniq1} yields
\begin{align*}
\frac 12 (k_n*\|v(\cdot)\|^2)'(t) +(h_n*Av(t), v(t)) & \le 0,\;\forall t\in (0,T].
\end{align*}
Let $g(t) = \frac 12 (k_n*\|v(\cdot)\|^2)'(t) +(h_n*Av(t), v(t))$, then $g(t)\le 0, \;\forall t\in (0,T]$. Noting that, the relation
$$
\frac 12 (k_n*\|v(\cdot)\|^2)'(t) = g_n(t):= g(t) - (h_n*Av(t), v(t))
$$
is equivalent to (see \cite[Lemma 2.4]{VZ})
\begin{equation}\label{th-uniq2}
\frac 12 \|v(t)\|^2 = \frac 1n g_n(t) + l*g_n(t), t\in (0,T],
\end{equation}
and the fact that $g_n(t)\to g(t) - (Av(t), v(t))$ as $n\to\infty$, for $t\in (0,T]$, we obtain
$$
\frac 12 \|v(t)\|^2 = l*[g(\cdot)-\|A^{\frac 12}v(\cdot)\|^2](t)\le 0, \; t\in (0,T].
$$
Thus $v=0$ and the proof is complete.
\end{proof}
\subsection{Regularity}
By using (\textsf K), the problem \eqref{le1}-\eqref{le2} can be transformed to the integral equation
\begin{equation*}
u(t) + l*Au(t) = u_0 + l*g(t),\; t\in [0,T].
\end{equation*}
This allows us to employ the resolvent theory in \cite{Pruss} for regularity analysis. Noting that the solution operator for the equation
\begin{equation}\label{eq-base}
u(t) + l*Au(t) = u_0, \; t\in [0,T],
\end{equation}
is given by $S(t)u_0 = u(t)$, where $S(t)$ is defined by \eqref{op-S}. We refer to $S(\cdot)$ as the resolvent family. 

We recall some notions and fact stated in \cite{Pruss}.
\begin{definition}
Let $l\in L^1_{loc}(\mathbb R^+)$ be a function of subexponential growth, i.e. $\displaystyle \int_0^\infty |l(t)|e^{-\epsilon t}dt<\infty$ for every $\epsilon>0$. 
\begin{itemize}
\item $l$ is said to be of positive type if $\displaystyle \text{\rm Re}\int_0^T (l*\varphi)(t)\bar\varphi(t)dt\ge 0$ for each $\varphi\in C(\mathbb R^+;\mathbb C)$ and $T>0$.
\item For given $m\in\mathbb N$, $l$ is called $m$-regular if there exists a constant $c>0$ such that
$$
|\lambda^n \hat l^{(n)}(\lambda)| \le c|\hat l(\lambda)|\;\;\text{ for all \rm Re}\lambda>0, 0\le n\le m,
$$
here $\hat l$ is the Laplace transform of $l$.
\end{itemize}
\end{definition}
It is easily seen that, if $l$ is nonnegative and nonincreasing on $(0,\infty)$, then $l$ is of positive type. Indeed, let $\varphi(t) = p(t) + iq(t)$, then
\begin{align*}
\text{\rm Re}\int_0^T (l*\varphi)(t)\bar\varphi(t)dt & = \int_0^T dt\int_0^t l(t-s)[p(s)p(t) + q(s)q(t)]ds\\
& \ge l(T) \int_0^T [P(t)p(t) + Q(t)q(t)]dt\\
& = \frac 12 l(T) [P^2(T) + Q^2(T)],
\end{align*}
where $P=1*p$ and $Q=1*q$, the primitive of $p$ and $q$, respectively.
\begin{definition}
Equation \eqref{eq-base} is called parabolic if the following conditions hold:
\begin{enumerate}
\item $\hat l(\lambda)\neq 0$, $1/\hat l(\lambda)\in\rho (-A)$ for all $\text{\rm Re}\lambda\ge 0$.
\item There is a constant $M\ge 1$ such that $U(\lambda) = \lambda^{-1}(I+\hat l(\lambda)A)^{-1}$ satisfies
$$
\|U(\lambda)\| \le \frac{M}{|\lambda|}\;\;\text{ for all \rm Re}\lambda>0.
$$
\end{enumerate}
\end{definition}
We have the following sufficient condition for \eqref{eq-base} to be parabolic.
\begin{proposition}{\cite[Corollary 3.1]{Pruss}}\label{pp-parabolic}
Assume that $l\in L^1_{loc}(\mathbb R^+)$ is of subexponential growth and of positive type. If $-A$ generates a bounded analytic semigroup in $H$, then \eqref{eq-base} is parabolic.
\end{proposition}
Let us mention that, by the assumption (\textsf A), $-A$ generates a contraction $C_0$-semigroup in $H$, which is given by
$$
e^{-tA}v = \sum_{n=1}^\infty e^{-t\lambda_n}(v,e_n) e_n,\; v\in H.
$$
So the semigroup $\{e^{-tA}\}_{t\ge 0}$ is analytic due to \cite[Corollary 7.1.1]{Vrabie}.

The following result on the resolvent family for \eqref{eq-base} plays an important role in our analysis.
\begin{proposition}{\cite[Theorem 3.1]{Pruss}}\label{pp-reg}
Assume that \eqref{eq-base} is parabolic and the kernel function $l$ is $m$-regular for some $m\ge 1$. Then there is a resolvent family $S(\cdot)\in C^{(m-1)}((0,\infty);\mathcal L(H))$ for \eqref{eq-base}, and a constant $M\ge 1$ such that
$$
\|t^nS^{(n)}(t)\| \le M,\;\text{ for all } t> 0, n\le m-1.
$$
\end{proposition}
In order to obtain the differentiability of the resolvent family, we replace (\textsf K) by a stronger assumption.
\begin{itemize}
\it
\item[(\textsf K*)] The assumption (\textsf K) is satisfied with $l$ being 2-regular, nonincreasing and of subexponential growth.
\end{itemize}
\begin{remark}\label{rm-K}
As mentioned in \cite{CN1981}, the assumption (\textsf K) does not guarantee that $l$ is nonincreasing. However if $k$ is positive, decreasing, log-convex ($\ln k$ is a convex function), and $k(0+)=\infty$, then $l$ is nonincreasing.
\end{remark}
Employing Proposition \ref{pp-reg}, we have the following statement.
\begin{lemma}\label{lm-reg}
Let (\textsf A) and (\textsf K*) hold. Then the resolvent family $S(\cdot)$ defined by \eqref{op-S} is differentiable on $(0,\infty)$, the relation
\begin{equation}\label{lm-reg1} 
S'(t) =- AR(t),\; t\in (0,\infty),
\end{equation}
and the estimate
\begin{equation}\label{lm-reg2}
\|S'(t)\| \le \frac{M}{t},\; t\in (0,\infty),
\end{equation}
hold for some $M\ge 1$.
\end{lemma}
\begin{proof}
Since the kernel function $l$ is nonnegative and nonincreasing, it is of positive type. In addition, the assumption (\textsf A) ensures that $-A$ generates a bounded analytic semigroup. So \eqref{eq-base} is parabolic, according to Proposition \ref{pp-parabolic}. Therefore, it follows from Proposition \ref{pp-reg} that $S(\cdot)$ is differentiable on $(0,\infty)$ and estimate \eqref{lm-reg2} takes place. Finally, it is deduced from the formulation of $S$ and $R$ given by \eqref{op-S}-\eqref{op-R} that
\begin{align*}
S'(t)v & = \sum_{n=1}^\infty \partial_t s(t,\lambda_n)(v,e_n)e_n\\
& = \sum_{n=1}^\infty -\lambda_n r(t,\lambda_n) (v,e_n)e_n = -AR(t)v, \; t>0, v\in H,
\end{align*}
thanks to \eqref{eq-sr1}, which proves \eqref{lm-reg1}. 
\end{proof}
Denote by $C^\gamma([a,b];H)$, $\gamma\in (0,1)$, the space of H\"older continuous functions on $[a,b]$, that is, $f\in C^\gamma([a,b];H)$ iff
$$
\|f\|_{C^\gamma} = \sup_{t_1,t_2\in [a,b]}\frac{\|f(t_1)-f(t_2)\|}{|t_1-t_2|^\gamma}<\infty.
$$
\begin{theorem}\label{th-reg}
Let the hypotheses of Lemma \ref{lm-reg} hold. Assume that the function $g$ in \eqref{le1} belongs to $C^\gamma([0,T];H)$, and $u$ is the weak solution of \eqref{le1}-\eqref{le2}. Then $u\in C([0,T];H)\cap C^\gamma([\delta,T];H)$ for any $0<\delta<T$, and $u$ is a classical solution.
\end{theorem}
\begin{proof}
Recall that the unique weak solution of \eqref{le1}-\eqref{le2} is given by
\begin{equation}\label{th-reg0}
u(t) = S(t)u_0 + (R*g)(t) = u_1(t) + u_2(t),\; t\in [0,T].
\end{equation}
We first show that $u_2$ is H\"older continuous on $[\delta,T]$. Indeed, for $t\in [\delta,T), h>0$, we have
\begin{align*}
\|u_2(t+h)-u_2(t)\| & \le \int_0^t \|R(\tau)\|\|g(t+h-\tau)-g(t-\tau)\|d\tau \\
&\quad + \int_t^{t+h}\|R(\tau)\|\|g(t+h-\tau)\|d\tau\\
& = I_1 + I_2.
\end{align*}
Considering $I_1$, one gets
\begin{align*}
I_1 & \le \int_0^t r(\tau,\lambda_1)\|g\|_{C^\gamma}h^\gamma d\tau = \|g\|_{C^\gamma}h^\gamma \lambda_1^{-1}(1-s(t,\lambda_1))\\
& \le \|g\|_{C^\gamma}\lambda_1^{-1}h^\gamma.
\end{align*}
Concerning $I_2$, the relation $S'(t) = -AR(t)$ for $t>0$ implies
\begin{align*}
I_2 & \le \|(-A)^{-1}\|\int_t^{t+h}\|S'(\tau)\|\|g(t+h-\tau)\|d\tau\\
& \le \|A^{-1}\|\|g\|_\infty M \int_t^{t+h}\frac{d\tau}{\tau} = \|A^{-1}\|\|g\|_\infty M \ln\left(1+\frac h t\right)\\
& \le \|A^{-1}\|\|g\|_\infty M \gamma^{-1} \left(\frac{h}{t}\right)^\gamma\\
& \le \|A^{-1}\|\|g\|_\infty M \gamma^{-1}\delta^{-\gamma} h^\gamma,
\end{align*}
here we utilize the inequality
$$
\ln(1+r)\le \frac{r^\gamma}{\gamma}\;\text{ for } r>0, \gamma \in (0,1).
$$
So we have proved that $\|u_2(t+h)-u_2(t)\|\le Ch^\gamma$ with $$C = \|g\|_{C^\gamma}\lambda_1^{-1}+\|A^{-1}\|\|g\|_\infty M \gamma^{-1}\delta^{-\gamma}.$$
It remains to show that $u_1\in C^\gamma([\delta,T];H)$. Let $0<\delta\le t<T$ and $h>0$. Using the mean value formula
$$
S(t+h)v - S(t)v= h\int_0^1 S'(t+\theta h)vd\theta,\; v\in H,
$$
we have
\begin{align*}
\|u_1(t+h) - u_1(t)\| & = \|S(t+h)u_0 - S(t)u_0\|\\
& \le h\int_0^1 \|S'(t+\theta h\|\|v\|d\theta\\
& \le M \|v\| h \int_0^1 \frac{d\theta}{t+\theta h} = M \|v\|\ln \left(1+\frac h t\right)\\
& \le M \|v\| \gamma^{-1}\delta^{-\gamma} h^\gamma.
\end{align*}
Finally, we have to show that $u$ is classical, that is, $u$ given by \eqref{th-reg0} obeys the system \eqref{le1}-\eqref{le2} in $H$. In the proof of Theorem \ref{th-wsol}, we have testified that \eqref{le1} holds in $V_{-\frac 12}$ by reasoning that $A(R*g)(t)\in V_{-\frac 12}$ for $t>0$. So it suffices to prove $A(R*g)(t)\in H$ for $t>0$ under the assumption that $g$ is H\"older continuous. Indeed, using the relation $S'(t) = -AR(t)$ for $t>0$ again, we obtain
\begin{align*}
A(R*g)(t) & = \int_0^t AR(t-\tau)g(\tau)d\tau = -\int_0^t S'(t-\tau)g(\tau)d\tau\\
& = -\int_0^t S'(t-\tau)[g(\tau)-g(t)]d\tau + [I-S(t)]g(t).
\end{align*}
Then
\begin{align*}
\|A(R*g)(t)\| & \le \int_0^t \|S'(t-\tau)\|\|g(\tau)-g(t)\|d\tau + \|[I-S(t)]g(t)\|\\
& \le M\|g\|_{C^\gamma}\int_0^t (t-\tau)^{\gamma-1}d\tau + \|[I-S(t)]g(t)\|\\
& \le M\|g\|_{C^\gamma}\gamma^{-1}T^\gamma + 2\|g\|_\infty,\; \text{ for } 0<t\le T,
\end{align*}
which completes the proof.
\end{proof}
\section{Stability and regularity for semilinear equations}
\begin{definition}
A function $u\in C([0,T];H)$ is called a mild solution of the problem \eqref{e1}-\eqref{e2} on $[0,T]$ iff
$$
u(t) = S(t)u_0 + \int_0^t R(t-s)f(u(s))ds,
$$
for every $t\in [0,T]$, where $S(\cdot)$ and $R(\cdot)$ are given by \eqref{op-S}-\eqref{op-R}. 
\end{definition}

In the next theorem, we prove a local solvability result.
\begin{theorem}\label{th-locsol}
Let (\textsf A), (\textsf K) and (\textsf F) be satisfied. Then there exists $t^*>0$ such that the problem \eqref{e1}-\eqref{e2} has a mild solution defined on $[0,t^*]$. Moreover, $u(t)\in V_{\frac 12}$ for all $t\in (0, t^*]$.
\end{theorem}
\begin{proof}
We make use of the contraction mapping principle. For given $\zeta \in (0,T]$, let $\Phi: C([0,\zeta];H)\to C([0,\zeta];H)$ be the mapping defined by
\begin{equation}\label{sol-op}
\Phi(u)(t) = S(t) u_0 + \int_0^t R(t-\tau)f(u(\tau))d\tau, t\in [0,\zeta].
\end{equation}
Taking $\rho>\|u_0\|$ and assuming that $u\in \mathbb B_\rho$, the closed ball in $C([0,\zeta];H)$ with center at origin and radius $\rho$, we have
\begin{align*}
\|\Phi(u)(t)\| & \le \|S(t)\|\|u_0\| + \int_0^t \|R(t-\tau)\|\|f(u(\tau))\| d\tau\\
& \le s(t,\lambda_1)\|u_0\| + \int_0^t r(t-\tau,\lambda_1)[\kappa(\rho)\|u(\tau)\| + \|f(0)\|]d\tau\\
& \le \|u_0\| + [\kappa(\rho)\rho + \|f(0)\|]\lambda_1^{-1}(1-s(t,\lambda_1)),\; t\in [0,\zeta],
\end{align*}
here we employ the hypothesis (\textsf{F}), Proposition \ref{pp-sr} and Lemma \ref{lm-SR}. Since $s(\cdot,\lambda_1)\in AC([0,\zeta])$ and $s(0,\lambda_1)=1$, one can choose $\zeta$ such that the last expression is smaller than $\rho$ as long as $t\in [0,\zeta]$. That is, $\Phi(\mathbb B_\rho)\subset \mathbb B_\rho$.

Using (\textsf{F}) again, one gets
\begin{align*}
\|\Phi(u_1)(t)-\Phi(u_2)(t)\| & \le \int_0^t r(t-\tau,\lambda_1)\|f(u_1(\tau)-f(u_2(\tau))\|d\tau\\
& \le \int_0^t r(t-\tau,\lambda_1)\kappa(\rho)\|u_1(\tau)-u_2(\tau)\|d\tau\\
& \le \kappa(\rho)\|u_1-u_2\|_\infty \lambda_1^{-1}(1-s(t,\lambda_1)),\; t\in [0,\zeta],
\end{align*}
where $\|\cdot\|_\infty$ is the sup norm in $C([0,\zeta];H)$. Taking $t^*\le \zeta$ such that $\kappa(\rho) (1-s(t^*,\lambda_1))<\lambda_1$, we see that $\Phi$ is a contraction as a map from $\mathbb B_\rho$ into itself, with $\mathbb B_\rho$ now in $C([0,t^*];H)$. So the problem \eqref{e1}-\eqref{e2} has a solution defined on $[0,t^*]$. In addition, since $t\mapsto g(t) = f(u(t))$ is a continuous function, $\Phi(u)(t)\in D(A^{\frac 12})$ for $t>0$ due to Remark \ref{rm-SR}. So $u(t)\in V_{\frac 12}$ for $t>0$. The proof is complete.

\end{proof}
We now discuss some circumstances, in which solutions exist globally.
\begin{theorem}\label{th-glosol1}
Let (\textsf A) and (\textsf K) hold. If the nonlinear function $f$ is globally Lipschitzian, that is, $\kappa(\rho)$ in (\textsf F) is constant, then the problem \eqref{e1}-\eqref{e2} has a unique global mild solution $u\in C([0,T];H)\cap C((0,T];V_{\frac 12})$. If, in addition, that $\kappa<\lambda_1$ and $l\not\in L^1(\mathbb R^+)$, then every mild solution to \eqref{e1} is globally bounded and asymptotically stable.
\end{theorem}
\begin{proof}
Fixed $T>0$, let $\beta>0$ and $\|u\|_\beta  = \sup\limits_{t\in [0,T]}e^{-\beta t}\|u(t)\|$. Then $\|\cdot\|_\beta$ is  equivalent to the sup norm in $C([0,T];H)$. From the estimate
\begin{align*}
\|\Phi(u_1)(t) - \Phi(u_2)(t)\| & \le \int_0^t r(t-\tau,\lambda_1)\kappa\|u_1(\tau)-u_2(\tau)\|d\tau,
\end{align*}
we get
\begin{align*}
e^{-\beta t}\|\Phi(u_1)(t) - \Phi(u_2)(t)\| & \le \left(\kappa \int_0^t r(t-\tau,\lambda_1)e^{-\beta (t-\tau)}d\tau\right) \|u_1-u_2\|_\beta\\
& \le \left(\kappa \int_0^T r(t,\lambda_1)e^{-\beta t}dt\right) \|u_1-u_2\|_\beta.
\end{align*}
Choosing $\beta>0$ such that 
$$
\kappa \int_0^T r(t,\lambda_1)e^{-\beta t}dt <1,
$$
we obtain $\Phi$ is a contraction map from $C([0,T];H)$ endowed with the norm $\|\cdot\|_\beta$ into itself, which ensures the existence and uniqueness of solution to \eqref{e1}-\eqref{e2}. In addition, we have $u(t)\in V_{\frac 12}$ for $t\in (0,T]$, by the same reasoning as in the proof of Theorem \ref{th-locsol}.

Now assume that $\kappa<\lambda_1$. Let $u$ be a solution of \eqref{e1}-\eqref{e2}, then we have
\begin{align*}
\|u(t)\|\le s(t,\lambda_1)\|u_0\| + \int_0^t r(t-\tau, \lambda_1)[\kappa \|u(\tau)\| + \|f(0)\|]d\tau,\;\forall t\ge 0.
\end{align*}
Using the Gronwall type inequality given in Proposition \ref{pp-gronwall}, we get
\begin{align*}
\|u(t)\| & \le s(t,\lambda_1-\kappa)\|u_0\| + \frac{1}{\lambda_1-\kappa}\|f(0)\| (1-s(t,\lambda_1-\kappa))\\
& \le \|u_0\| + \frac{1}{\lambda_1-\kappa}\|f(0)\|,\;\forall t\ge 0,
\end{align*}
which yields the global boundedness of $u$.

Let $u$ and $v$ be solutions of \eqref{e1}, then we have
\begin{align*}
\|u(t)-v(t)\| \le s(t,\lambda_1)\|u(0)-v(0)\| + \int_0^t r(t-\tau,\lambda_1)\kappa\|u(\tau)-v(\tau)\|d\tau,
\end{align*}
thanks to (\textsf F) and Lemma \ref{lm-SR1}. Employing Proposition \ref{pp-gronwall} again, we obtain
\begin{align*}
\|u(t)-v(t)\| \le s(t,\lambda_1-\kappa)\|u(0)-v(0)\|,\; \forall t\ge 0.
\end{align*}
Since $l\not\in L^1(\mathbb R^+)$, it follows from Proposition \ref{pp-sr}(1) that $s(t,\lambda_1-\kappa)\to 0$ as $t\to\infty$, which completes the proof.
\end{proof}
The following theorems show the main results of this section.
\begin{theorem}\label{th-glosol2}
Let (\textsf A), (\textsf K) and (\textsf F) hold. If $\lim\limits_{v\to 0}\dfrac{\|f(v)\|}{\|v\|} = \alpha$ with $\alpha\in [0,\lambda_1)$, then there exists $\delta>0$ such that the problem \eqref{e1}-\eqref{e2} admits a unique global mild solution $u\in C([0,T];H)\cap C((0,T];V_{\frac 12})$, provided that $\|u_0\|\le \delta$. 
\end{theorem}
\begin{proof}
In order to prove the global existence, we make use of the Schauder fixed point theorem. By assumption on the behaviour of $f$, for $\theta\in (0, \lambda_1-\alpha)$, there exists $\eta >0$ such that $\|f(v)\|\le (\alpha + \theta) \|v\| $ as long as $\|v\|\le \eta$. Now we consider the solution map $\Phi: \mathbb B_\eta\to C([0,T];H)$ defined by \eqref{sol-op}. We see that
\begin{align*}
\|\Phi(u)(t)\| & \le s(t,\lambda_1)\|u_0\| + \int_0^t r(t-\tau,\lambda_1)(\alpha+\theta)\|u(\tau)\|d\tau\\
& \le s(t,\lambda_1)\|u_0\| + (\alpha+\theta)\eta \lambda_1^{-1}(1-s(t,\lambda_1))\\
& \le s(t,\lambda_1)[\|u_0\| - (\alpha+\theta)\lambda_1^{-1}\eta ] + (\alpha+\theta)\lambda_1^{-1}\eta\\
& \le \eta, \; t\in [0,T],
\end{align*}
provided that $\|u_0\|\le \alpha\lambda_1^{-1}\eta$, thanks to the fact that $(\alpha+\theta)\lambda_1^{-1}<1$. Fixing an $\theta$ and $\eta$ mentioned above, for $\delta = \alpha\lambda_1^{-1}\eta$, we have shown that $\Phi(\mathbb B_\eta)\subset \mathbb B_\eta$ as $\|u_0\|\le \delta$.

In the next step, we construct a compact convex subset $\mathcal D\subset\mathbb B_\eta$ which is still invariant under $\Phi$ by using the same routine as in \cite{VAK}. Put $\mathcal M_0=\mathbb B_\eta$ and $\mathcal M_{k+1} = \overline{\text{co}}\,\Phi(\mathcal M_k), k\in\mathbb N$, where $\overline{\text{co}}$ denotes the closure of convex hull of subsets in $C([0,T];H)$. Obviously, $\mathcal M_k$ is closed, convex, and $\mathcal M_{k+1}\subset \mathcal M_k$ for all $k\in \mathbb N$. Let $\mathcal M=\bigcap\limits_{k=1}^\infty\mathcal M_k$, then $\mathcal M$ is also a closed convex set and $\Phi(\mathcal M)\subset \mathcal M$. We verify that $\mathcal M(t)$ is relatively compact for each $t\ge 0$. Indeed, let $\chi$ be the Hausdorff measure of noncompactness on $H$. Then it suffices to testify that $\lim\limits_{k\to\infty}\chi(\mathcal  M_k(t))=0$ for each $t\ge 0$. For $D\subset \mathbb B_\eta$, we get
\begin{align*}
\chi(f(D(t)))\le \kappa(\eta)\chi(D(t)),\;\forall t\ge 0,
\end{align*}
thanks to the Lipschitz property of $f$ (see \cite{KOZ}), here $D(t)=\{v(t): v\in D\}$. Therefore,
\begin{align*}
\chi(\mathcal M_{k+1}(t)) & = \chi(\Phi(\mathcal M_k)(t))\\
& \le \chi\left(\int_0^t R(t-\tau)f(\mathcal M_k(\tau))d\tau\right)\\
& \le \int_0^t \|R(t-\tau)\|\chi(f(\mathcal M_k (\tau)))d\tau\\
& \le \int_0^t r(t-\tau,\lambda_1)\kappa(\eta)\chi(\mathcal M_k (\tau))d\tau.
\end{align*}
Let $\mu_k(t) = \chi(\mathcal M_k (t)), t\ge 0$, then $\mu_k$ is nonincreasing and the last estimate reads
$$
\mu_{k+1}(t) \le \int_0^t r(t-\tau,\lambda_1)\kappa(\eta)\mu_k(\tau)d\tau.
$$
Passing to the limit in the last relation, one gets
$$
\mu_\infty (t) \le \int_0^t r(t-\tau,\lambda_1)\kappa(\eta)\mu_\infty (\tau)d\tau,
$$
where $\mu_\infty(t) = \lim\limits_{k\to\infty}\mu(t), t\ge 0$. Now applying Proposition \ref{pp-gronwall} yields $\mu_\infty(t) = 0$ for every $t\ge 0$. We have proved that $\mathcal M(t)$ is relatively compact for each $t\ge 0$. 

We are now in a position to define the set $\mathcal D=\overline{\text{co}}\,\Phi(\mathcal M)$. Since $\Phi (\mathcal M)\subset\mathcal M$ and $\mathcal M$ is closed and convex, we see that $\mathcal D\subset \mathcal M$, and then
$$
\Phi(\mathcal D)\subset \Phi(\mathcal M)\subset \overline{\text{co}}\Phi(\mathcal M)=\mathcal D.
$$
Moreover, since $\mathcal D(t)\subset \mathcal M(t)$, we have $\mathcal D(t)$ relatively compact for every $t\ge 0$. It remains to show that $\mathcal D$ is a equicontinuous subset, which is equivalent to that $\Phi(\mathcal M)$ has this property. Observe that, if $Y\subset H$ is a relatively compact set, then the set $R(\cdot)Y=\{R(\cdot)y: y\in Y\}$ is equicontinuous on $(0,T]$. The same reasoning ensures the set $R(\cdot-\tau)f(\mathcal M(\tau))$ is equicontinuous on $(\tau, T]$. This implies the equicontinuity of the set
$$
\Phi(\mathcal M)=\{v: v(t) \in S(t)u_0 + \int_0^t R(t-\tau)f(\mathcal M(\tau))d\tau\}.
$$
Thus $\mathcal D$ is compact due to the Arzela-Ascoli theorem. This enables us to utilize the Schauder fixed point theorem for $\Phi:\mathcal D\to\mathcal D$ to obtain a mild solution for \eqref{e1}-\eqref{e2}. 

Finally, if $u_1$ and $u_2$ is two mild solutions for \eqref{e1}-\eqref{e2}, $v=u_1-u_2$, and $\rho=\max\{\|u_1\|_\infty, \|u_2\|_\infty\}$ then
\begin{align*}
\|v(t)\| & \le \int_0^t \|R(t-\tau)\| \|f(u_1(\tau))-f(u_2(\tau))\|d\tau\\
& \le \int_0^t r(t-\tau,\lambda_1)\kappa(\rho)\|v(\tau)\|d\tau, \;t\in [0,T].
\end{align*}
Applying Proposition \ref{pp-gronwall} yields $v=0$. The proof is complete.
\end{proof}
\begin{theorem}\label{th-stab}
Let the hypotheses of Theorem \ref{th-glosol2} hold. If $l\not\in L^1(\mathbb R^+)$, then the zero solution of \eqref{e1}-\eqref{e2} is asymptotically stable.
\end{theorem}
\begin{proof}
Taken $\theta$ and $\delta$ from the proof of Theorem \ref{th-glosol2}, for $\|u_0\|\le \delta$ and a corresponding solution $u$ of \eqref{e1}-\eqref{e2}, we have
\begin{align*}
\|u(t)\|\le s(t,\lambda_1)\|u_0\| + \int_0^t r(t-\tau,\lambda_1)(\alpha+\theta)\|u(\tau)\|d\tau.
\end{align*}
Using Proposition \ref{pp-gronwall}, we get
$$
\|u(t)\|\le s(t, \lambda_1-\alpha-\theta)\|u_0\|, \;\forall t\ge 0.
$$
Since $l\not\in L^1(\mathbb R^+)$ and $\lambda_1-\alpha-\theta>0$, we have $s(t, \lambda_1-\alpha-\theta)\to 0$ as $t\to\infty$, and the last inequality ensures the stability and attractivity of the zero solution. The proof is complete.
\end{proof}

We now present a linearized stability result as a consequence of Theorem \ref{th-stab}.
\begin{corollary}
Let (\textsf{A}) and (\textsf{K}) hold. Assume that the nonlinearity $f$ is continuously differentiable such that $f(0)=0$ and $A-f'(0)$ remains positively definite. Then the zero solution of \eqref{e1} is asymptotically stable.
\end{corollary}
\begin{proof}
Denote $\tilde A = A - f'(0)$, $\tilde f(v) = f(v) - f'(0)v$. Then equation \eqref{e1} is equivalent to
\begin{equation}\label{ee}
\frac{d}{dt}[k*(u-u_0)](t) + \tilde A u(t) = \tilde f (u(t)), \; t>0.
\end{equation}
By assumption, $\tilde A$ fulfills (\textsf{A}). Furthermore, $\tilde{f}$ is also continuously differentiable, so it is locally Lipschitzian and, therefore, $\tilde f$ satisfies (\textsf F). Observing that $\|\tilde f(v)\| = o(\|v\|)$ as $\|v\|\to 0$, one can apply Theorem \ref{th-stab} for \eqref{ee} (with $\alpha=0$) to get the conclusion.
\end{proof}
To end this section, we prove the H\"older continuity of the mild solution to \eqref{e1}-\eqref{e2}.
\begin{theorem}\label{th-reg-sm}
Let (\textsf A), (\textsf K*) and (\textsf F) hold. Then the mild solution to \eqref{e1}-\eqref{e2} is H\"older continuous on $[\delta, T]$ for every $0<\delta<T$. 
\end{theorem}
\begin{proof}
Let $u$ be the mild solution to \eqref{e1}-\eqref{e2}. Then
\begin{align*}
u(t) & = S(t)u_0 + \int_0^t R(t-\tau)f(u(\tau))d\tau\\
& = u_1(t) + u_2(t).
\end{align*}
By the same reasoning as in the proof of Theorem \ref{th-reg}, we have $u_1\in C^\gamma([\delta, T];H)$ for every $0<\delta<T$ and $\gamma \in (0,1)$.

Regarding $u_2$, let $\rho=\|u_2\|_\infty$ and $0<\delta\le t \le T$, then we see that
\begin{align*}
\|u_2(t+h) - u_2(t)\| & \le \int_0^t \|R(\tau)\|\|f(u_2(t+h-\tau)) - f(u_2(t-\tau))\|d\tau \\
& \quad + \int_t^{t+h}\|R(\tau)\|\|f(u_2(t+h-\tau))\|d\tau\\
& \le \int_0^t r(\tau, \lambda_1)\kappa(\rho)\|u_2(t+h-\tau)-u_2(t-\tau)\|d\tau\\
& \quad + \|A^{-1}\|\int_t^{t+h}\|S'(\tau)\| (\|f(0)\| + \kappa(\rho)\rho)d\tau\\
& \le \int_0^t r(t-\tau, \lambda_1)\kappa(\rho)\|u_2(\tau+h)-u_2(\tau)\|d\tau\\
& \quad + \|A^{-1}\| M (\|f(0)\| + \kappa(\rho)\rho) \gamma^{-1} \delta^{-\gamma}h^\gamma,
\end{align*}
here we use (\textsf F) and the arguments as in the proof of Theorem \ref{th-reg} for estimating the second integral.

Applying Proposition \ref{pp-gronwall} for $v(t) = \|u_2(t+h) - u_2(t)\|$, one gets
\begin{align*}
\|u_2(t+h) - u_2(t)\| & \le \|A^{-1}\| M (\|f(0)\| + \kappa(\rho)\rho) \gamma^{-1} \delta^{-\gamma}s(t,\lambda_1-\kappa(\rho)) h^\gamma,
\end{align*}
which implies $u_2\in C^\gamma([\delta, T];H)$.
\end{proof}
\section{Application}
Let $\Omega\subset\mathbb R^N$ be a bounded domain with smooth boundary $\partial\Omega$. We apply the obtained results to the following two-term fractional-in-time PDE:
\begin{align}
\partial_t^\alpha u(t,x) + \mu\,\partial_t^\beta u(t,x) + (- \Lambda)^\gamma u(t,x) = &\; F\left(\int_\Omega u^2(t,x)dx\right)G(x,u(t,x)),\label{ap1}\\
& \;\text{ for } t>0, x \in\Omega,\notag\\
u(t,x)  = &\; 0,\; \text{ for } t\ge 0,\;  x\in\partial\Omega, \label{ap2}\\
u(0,x)  = &\; u_0(x),\; \text{ for } x\in\Omega,\label{ap3}
\end{align}
where $0<\alpha<\beta<1$, $\mu\ge 0, \gamma>0$, $\partial_t^\alpha$ and $\partial_t^\beta$ stand for the Caputo fractional derivatives of order $\alpha$ and $\beta$ in $t$, respectively. The operator $\Lambda$ is defined by
$$
D(\Lambda) = H^2(\Omega)\cap H_0^1(\Omega),\;\Lambda u = \sum_{i,j=1}^N\partial_{x_i}(a_{ij}(x)\partial_{x_j} u), 
$$
where $a_{ij}\in L^\infty(\Omega), a_{ij}=a_{ji}, 1\le i,j\le N$, subject to the condition $\sum\limits_{i,j=1}^N a_{ij}(x)\xi_i\xi_j\ge \theta |\xi|^2$, for some $\theta >0$. Let $H=L^2(\Omega)$ with the inner product $\displaystyle (u,v)=\int_\Omega u(x)v(x)dx$. Put
\begin{align}
k(t) & = g_{1-\alpha}(t) + \mu\,g_{1-\beta}(t),\label{ap-k}\\
A & = (-\Lambda)^\gamma,\notag\\
f(v)(x) & = F\left(\int_\Omega v^2(x)dx\right)G(x,v(x)), \; v\in L^2(\Omega).\notag
\end{align}
Then the problem \eqref{ap1}-\eqref{ap3} is in the form of \eqref{e1}-\eqref{e2}. Observe that, the kernel function $k$ is completely monotonic, i.e. $(-1)^n k^{(n)}(t)\ge 0$ for $t\in (0,\infty)$. As mentioned in \cite{VZ}, $k$ admits a resolvent function $l$ such that $k*l=1$ on $(0,\infty)$ and in this case, $(1*l)(t)\sim g_{1+\alpha}(t)$ as $t\to\infty$. Thus
$$
s(t,\lambda_1) \le \frac{1}{1+\lambda_1 (1*l)(t)}\to 0\text{ as } t\to\infty.
$$
Noting that, the nonlinearity in \eqref{ap1} can be seen as a perturbation depending not only on the state but also on the energy of the system.  We assume that 
\begin{itemize}
\item $F\in C^1(\mathbb R)$ obeys the estimate $|F(r)|\le a + b|r|^\nu$, for some nonnegative numbers $a, b$ and $\nu$.
\item $G: \Omega\times\mathbb R\to \mathbb R$ is a Carath\'eodory function and satisfies the Lipschitz condition in the second variable, i.e.
$$
|G(x,y_1)-G(x,y_2)|\le h(x)|y_1-y_2|, \forall x\in\Omega, y_1, y_2\in\mathbb R,
$$
here $h\in L^\infty (\Omega)$ is a nonnegative function. In addition, assume that $G(x,0)=0$ for a.e. $x\in\Omega$.
\end{itemize}
Then one can verify that $f$ maps $L^2(\Omega)$ into itself. More precisely, we get
\begin{align*}
\|f(v)\| & = F\left(\int_\Omega v^2(x)dx\right)\left(\int_\Omega |G(x,v(x))|^2dx\right)^{\frac 12}\\
& \le F\left(\int_\Omega v^2(x)dx\right)\left(\int_\Omega h^2(x)v^2(x)dx\right)^{\frac 12}.
\end{align*}
Hence
\begin{align}\label{ap4}
\|f(v)\| \le (a+b\|v\|^{2\nu})\|h\|_\infty \|v\|,
\end{align}
where $\|h\|_\infty = \text{ess\,sup}_{x\in\Omega} |h(x)|$.
In addition, we can check that $f$ is locally Lipschitzian due to the assumption that $F'$ is continuous and $G$ is Lipschitzian. Specifically, for $v_1, v_2\in L^2(\Omega), \|v_1\|, \|v_2\|\le \rho$, we see that
\begin{align*}
\|f(v_1)-f(v_2)\| & \le |F(\|v_1\|^2) - F(\|v_2\|^2)| \left(\int_\Omega |G(x,v_1(x))|^2 dx\right)^{\frac 12}\\
& \quad + |F(\|v_2\|^2)|\left(\int_\Omega |G(x,v_1(x))-G(x,v_2(x))|^2dx\right)^{\frac 12}\\
& \le |F'\left(\theta\|v_1\|^2 + (1-\theta)\|v_2\|^2\right)|\cdot |\|v_1\|^2-\|v_2\|^2|\cdot \|h\|_\infty \|v_1\|\\
& \quad + (a+b\|v_2\|^{2\nu})\|h\|_\infty \|v_1-v_2\|\\
& \le \kappa(\rho)\|v_1-v_2\|,
\end{align*}
where 
$$
\kappa(\rho) = 2\rho^2\|h\|_\infty\sup_{r\in [0,\rho^2]} |F'(r)| + (a+b\rho^{2\nu})\|h\|_\infty.
$$
Let $\lambda_\triangle$ be the first eigenvalue of $-\Delta$, that is $\lambda_\triangle = \inf\{\|\nabla u\|^2: \|u\|=1\}$. Denote by $\mu_1$ the first eigenvalue of $-\Lambda$ associated with the eigenfunction $\varphi$, then
\begin{align*}
\mu_1\|\varphi\|^2 = (-\Lambda \varphi, \varphi)\ge \theta \|\nabla \varphi\|^2 \ge \theta \lambda_\triangle \|\varphi\|^2,
\end{align*}
thanks to the Poincar\'e inequality. This implies that the first eigenvalue $\lambda_1$ of $A=(-\Lambda)^\gamma$ is given by $\lambda_1 = \mu_1^\gamma \ge  \theta^\gamma \lambda_\triangle^\gamma$. 

On the other hand, it follows from \eqref{ap4} that $\lim\limits_{v\to 0}\dfrac{\|f(v)\|}{\|v\|} \le a\|h\|_\infty$ if $\nu>0$ and $\lim\limits_{v\to 0}\dfrac{\|f(v)\|}{\|v\|}\le (a+b)\|h\|_\infty$ as $\nu=0$. Employing Theorem \ref{th-stab}, we have the conclusion that, the zero solution of \eqref{ap1} is asymptotically stable in the following cases:
\begin{enumerate}
\item $\nu>0$ and $a\|h\|_\infty<\theta^\gamma \lambda_\triangle^\gamma$, 
\item $\nu=0$ and $(a+b)\|h\|_\infty<\theta^\gamma \lambda_\triangle^\gamma$.
\end{enumerate}
Let us mention that, the mild solution for \eqref{ap1}-\eqref{ap4} is H\"older continuous on $[\delta, T]$ for every $0<\delta<T$. Indeed, since $k$ given by \eqref{ap-k} is positive, decreasing and log-convex on $(0,\infty)$, by Remark \ref{rm-K}, the kernel function $l$ is nonincreasing. Moreover, the Laplace transform $\hat l$ is given by $\hat l(\lambda) = (\lambda^{1-\alpha} + \mu \lambda^{1-\beta})^{-1}$. By a direct computation, one has
\begin{align*}
\lambda \hat l'(\lambda) & = - (\lambda^{1-\alpha} + \mu \lambda^{1-\beta})^{-2}[(1-\alpha)\lambda^{1-\alpha} + \mu (1-\beta)\lambda^{1-\beta}]\\
\lambda^2 \hat l''(\lambda) & = 2(\lambda^{1-\alpha} +\mu \lambda^{1-\beta})^{-3}[(1-\alpha)\lambda^{1-\alpha} + \mu (1-\beta)\lambda^{1-\beta}]^2 \\
& \qquad + (\lambda^{1-\alpha} + \mu \lambda^{1-\beta})^{-2}[\alpha(1-\alpha)\lambda^{1-\alpha} + \mu \beta(1-\beta)\lambda^{1-\beta}].
\end{align*}
Observing that, for every $\eta_1, \eta_2\in (0,1)$ and Re$\lambda>0$, 
$$
|\eta_1 \lambda^{1-\alpha} + \eta_2\mu \lambda^{1-\beta}|\le |\lambda^{1-\alpha} + \mu \lambda^{1-\beta}|,
$$
we have
\begin{align*}
| \lambda \hat l'(\lambda) | \le |\hat l(\lambda) |,\;\; | \lambda^2 \hat l''(\lambda) | \le 3|\hat l(\lambda) |,
\end{align*}
which ensures that $l$ is 2-regular, and (\textsf K*) is fulfilled. So the H\"older regularity of the mild solution follows from Theorem \ref{th-reg-sm}.


\begin{thebibliography}{00}

\bibitem{AK} N.T. Anh, T.D. Ke, \textit{Decay integral solutions for neutral fractional differential equations with infinite delays}, Math. Methods Appl. Sci. 38 (2015), 1601-1622.

\bibitem{VAK}  N.T.V. Anh, T.D. Ke, \textit{On the differential variational inequalities of parabolic-elliptic type}, Math. Methods Appl. Sci. 40 (2017), 4683-4695. 

% ref: nonlocal
\bibitem{CN1981}  Ph. Cl\'ement, J. A. Nohel,  \textit{Asymptotic behavior of solutions of nonlinear Volterra equations with completely positive kernels}, SIAM J. Math. Anal. 12 (1981), 514-535.

\bibitem{Grip1985} G. Gripenberg, \textit{Volterra integro-differential equations with accretive nonlinearity}, J. Differential Equations 60 (1985), 57-79.

\bibitem{JW2004} V.G. Jakubowski, P. Wittbold, \textit{On a nonlinear elliptic-parabolic integro-differential equation with $L^1$-data}, J. Differential Equations 197 (2004), 427-445.

\bibitem{Koc} Anatoly N. Kochubei, \textit{General Fractional Calculus, Evolution Equations, and Renewal Processes}, Integr. Equ. Oper. Theory 71 (2011), 583-600

\bibitem{KOZ} M. Kamenskii, V. Obukhovskii, P. Zecca, Condensing multivalued maps and semilinear differential inclusions in Banach spaces, Walter de Gruyter, Berlin, New York, 2001.

\bibitem{KL} T.D. Ke, D. Lan, \textit{Fixed point approach for weakly asymptotic stability of fractional differential inclusions involving impulsive effects}, J. Fixed Point Theory Appl. 19 (2017), 2185-2208.

\bibitem{KL2} T.D. Ke, D. Lan, \textit{Decay integral solutions for a class of impulsive fractional differential equations in Banach spaces}, Fract. Calc. Appl. Anal. 17:1 (2014), 96-121.

\bibitem{KSVZ} J. Kemppainen, J. Siljander, V. Vergara, R. Zacher, \textit{Decay estimates for time-fractional and other non-local in time subdiffusion equations in $R^d$}, Math. Ann. 366 (2016), 941-979.

\bibitem{Miller} R.K. Miller,  \textit{On Volterra integral equations with nonnegative integrable resolvents}, J. Math. Anal. Appl. 22 (1968), 319-340.

\bibitem{SC} S.G. Samko, R.P. Cardoso, \textit{Integral equations of the first kind of Sonine type}, Int. J. Math. Math. Sci. 57 (2003), 3609-3632.

\bibitem{Pruss} J. Pr\"uss, Evolutionary Integral Equations and Applications. Monographs in Mathematics 87,
Birkh\"auser, Basel, 1993.

\bibitem{VZ} V. Vergara, R. Zacher, \textit{Optimal decay estimates for time-fractional and other nonlocal subdiffusion equations via energy methods}, SIAM J. Math. Anal. 47 (2015), 210-239.

\bibitem{VZ2017} V. Vergara, R. Zacher, \textit{Stability, instability, and blowup for time fractional and other nonlocal in time semilinear subdiffusion equations}, J. Evol. Equ. 17 (2017), 599-626.

\bibitem{Vrabie} I.I. Vrabie, $C_0$-Semigroups and Applications. North-Holland Publishing Co., Amsterdam, 2003.

\bibitem{Zacher2} R. Zacher, \textit{Boundedness of weak solutions to evolutionary partial integro-differential equations
with discontinuous coefficients}, J. Math. Anal. Appl. 348 (2008),  137-149.
%ref: to compare
\bibitem{Zacher} R. Zacher, \textit{Weak solutions of abstract evolutionary integro-differential equations in Hilbert spaces}, Funkcialaj Ekvacioj 52 (2009), 1-18.

\end{thebibliography}
\end{document}